\documentclass[10pt]{article} 
\usepackage{amssymb}
\usepackage{xcolor}
\usepackage{amsmath,amsfonts,amsthm}
\usepackage{graphicx}
\usepackage{comment}
\usepackage[pdftex]{hyperref}
\newtheorem{coro}{Corollary}
\newtheorem{defin}{Definition}
\newtheorem{theorem}{Theorem}[section]
\newtheorem{lemma}[subsection]{Lemma}
\newtheorem{proposition}[theorem]{Proposition}

\newtheorem{remark}[theorem]{Remark}
\newtheorem{definition}[theorem]{Definition}

\newcommand{\abs}[1]{\left\vert #1 \right\vert}
\newtheorem{theoremletter}{Theorem}

\DeclareMathOperator{\Ricc}{Ric}

\usepackage{fullpage}

\begin{document}

\numberwithin{equation}{section}

\title{Rigidity results for Serrin’s overdetermined problems in Riemannian manifolds}


\author{Maria Andrade$^1$, Allan Freitas$^2$, and Diego A. Marín$^3$}
\footnotetext[1]{Departamento de Matem\'atica, Universidade Federal de Sergipe, DMA,
	 s/n, CEP 49100-000, S\~ao Crist\'ov\~ao, SE, Brazil. \textsf{maria@mat.ufs.br}}

\footnotetext[2]{
Departamento de Matem\'atica, Universidade Federal da Para\'iba, Cidade Universit\'aria, 58051-900, João Pessoa, Para\'iba, Brazil. \textsf{allan.freitas@academico.ufpb.br}}

\footnotetext[3]{ Department of Geometry and Topology and Institute of Mathematics (IMAG), University of
Granada, 18071, Granada, Spain. \textsf{damarin@ugr.es}}

\date{}

\maketitle{}
\begin{abstract}
In this work, we are interested in studying Serrin’s overdetermined problems in Riemannian manifolds. For manifolds endowed with a conformal vector field, we prove a Pohozaev-type identity to establish a Serrin-type rigidity result using the P-function approach introduced by Weinberger. To achieve this, we proceed with a conformal change, starting from a geometric identity due to Schoen. Moreover, we obtain a symmetry result for the associated Dirichlet problem by utilizing a generalized normalized wall shear stress bound.
\end{abstract}

\vspace{0.2cm} \noindent \emph{2020 Mathematics Subject
Classification} : 35R01; 35N25; 53C24

\vspace{0.4cm}\noindent \emph{Keywords}: Serrin's type problem, Rigidity, $P$-function, Conformal fields.


\section{Introduction}
\indent \par In 1971, Serrin \cite{serrin} began studying the following problem:
\begin{eqnarray}
\Delta u  &=& -1 \ \ \hbox{in}\quad \Omega, \quad u=0\quad \hbox{on}\quad  \partial\Omega, \label{serrinok1}\\
|\nabla u|&=&c\quad \hbox{on}\quad  \partial\Omega,\label{serrinok2}    
\end{eqnarray}
where $\Omega \subset \mathbb{R}^n$ is a bounded domain with boundary $\partial \Omega$ of class $C^2$, and $c$ is a positive constant.

When studying a system such as the one described above, where both Dirichlet data \eqref{serrinok1} and Neumann data \eqref{serrinok2} are prescribed, one encounters a so-called overdetermined elliptic problem (OEP). This type of problem frequently arises in physical models and free boundary problems (see \cite{EMa,EspMaz,rayleigh}). Solutions to such problems are expected to be rare. A notable solution in this context is the Euclidean ball $B(0, R)$, with the radial function
\begin{equation}\label{class_solution}
u(x)=\frac{R^2-|x|^2}{2n}.    
\end{equation}
In fact, Serrin proved that the only solution to \eqref{serrinok1}-\eqref{serrinok2}, up to translations, is this one, using the method of moving planes. Serrin's moving planes method was inspired by the proof of Alexandrov’s Theorem (see \cite{Alexandrov}). Since Serrin's seminal work, the radial symmetry of solutions and their respective domains has been thoroughly investigated for certain overdetermined elliptic systems (see \cite{BNST,ciraolo, garofalo, kumpraj, NT, SS, wein} and references therein, for example).
It is worth mentioning that this overdetermined system has its origins in fluid mechanics, specifically in a problem related to the flow of viscous incompressible fluids through pipes, as first studied by Serrin \cite[Page 1]{serrin}. The physical interpretation of terms such as the wall shear stress (WSS) arises from this context. For further details regarding the terminology and its physical significance, we refer the reader to the work by Agostiniani et al \cite{ABM}.

A natural question arises immediately in this context: how can we approach similar rigidity results for solutions of overdetermined problems like this in Riemannian manifolds? In this setting, the moving plane method works in manifolds with a certain symmetry, such as the space forms, where a rigidity result in this same sense could be obtained (see \cite{kumpraj}). However, it is not expected that the method of moving planes will work in general Riemannian manifolds.

Luckily, in the same edition of the above-mentioned Serrin's paper, Weinberger \cite{wein} provided a simpler proof (which also inspired some extensions, as we will mention shortly) using an integral technique. In short, he defined a so-called \textit{$P$-function}, \( P(u) = |\nabla u|^2 + \frac{2}{n}u \), associated with the solution \( u \) of \eqref{serrinok1}-\eqref{serrinok2}. The \( P \)-function satisfies \( P(u) \equiv c^2 \) on \( \partial \Omega \) and is subharmonic; more precisely, \( \Delta P(u) = 2|\mathring{\nabla}^2 u|^2 \geq 0 \), where $\mathring{\nabla}^2 u = {\nabla}^2 u - \left(\frac{\Delta u}{n}\right)I$, and $I$ represents the identity matrix of size $n \times n$. The strong maximum principle applied to \( P \) implies that either \( P(u) \equiv c^2 \) in \( \overline{\Omega} \) or \( P(u) < c^2 \) in \( \Omega \). The second case would lead to a contradiction with the \textit{classical Pohozaev identity} applied to \eqref{serrinok1}-\eqref{serrinok2} (see, e.g., Struwe~\cite[page~171]{struwe}). This means that \( P(u) \equiv c^2 \) in \( \overline{\Omega} \) and, since \( \mathring{\nabla}^2 u = 0 \), \( u \) is a radial function and \( \Omega \) is a ball. This approach has inspired many works on more general elliptic partial equations (see, for example, \cite{ciraolo, fragala, gao, garofalo}).

We remark that the main ingredients behind Weinberger's \( P \)-function approach are a maximum principle for a suitable \( P \)-function and a Pohozaev-type identity. These two tools have been identified as essential for studying Serrin-type problems in the context of Riemannian manifolds. In this direction, Ciraolo and Vezzoni \cite{ciraolo} used the \( P \)-function approach to prove the analog of Serrin's theorem for the problem
\begin{eqnarray}
\Delta u + nku  &=& -1 \quad \text{in} \quad \Omega, \quad u = 0 \quad \text{on} \quad \partial \Omega, \label{serrinoka}\\
|\nabla u| &=& c \quad \text{on} \quad \partial \Omega, \label{serrinokb}    
\end{eqnarray}
where \( \Omega \subset M \) is a bounded domain and \( M = \mathbb{H}^n, \mathbb{R}^n, \mathbb{S}^n_+ \), when \( k = -1, 0, 1 \), respectively. Farina and Roncoroni \cite{farina-roncoroni} studied this same problem on warped products, obtaining rigidity results by considering such ambients with Ricci curvature bounded below. More recently, Roncoroni, Santos, and the second named author \cite{freitas} have studied the Serrin-type problem \eqref{serrinoka}--\eqref{serrinokb} on manifolds endowed with a closed conformal vector field. We mention that, in these works, the natural \( P \)-function is given by
\begin{eqnarray}\label{pfunction}
P(u) = |\nabla u|^2 + \dfrac{2}{n}u + ku^2,
\end{eqnarray}
which is subharmonic when \( u \) solves \eqref{serrinoka}--\eqref{serrinokb} and \( Ric_g \geq (n-1)kg \), as a consequence of the Bochner identity (see \cite[Lemma 5]{farina-roncoroni}).

The primary objective of this work is to extend existing rigidity results for a specific class of manifolds endowed with a conformal vector field. Let \( (M^n, g) \) denote a Riemannian manifold. A vector field \( X \) is called a conformal vector field if its flow consists of conformal transformations. Equivalently, there exists a smooth function \( \psi \) on \( M \), referred to as the conformal factor of \( X \), such that the Lie derivative of the metric with respect to \( X \) satisfies the equation \( \mathcal{L}_X g = 2\psi g \). In this context, we establish the following result:
\begin{theoremletter}\label{theorem-A}
    Let \( (M^n, g) \) be a Riemannian manifold satisfying \( \operatorname{Ric}_g \geq (n-1)kg \), and endowed with a conformal vector field \( X \) with a positive (or negative) conformal factor \( \psi \). If \( \Omega \subset M \) is a bounded domain and \( u \) is a solution of \eqref{serrinoka}--\eqref{serrinokb}, satisfying the following compatibility condition:
    \begin{equation}\label{main1}
    \int_{\Omega} u^2 \left[ \psi(-R_g + n(n-1)k) - \frac{1}{2}X(R_g) \right]dV_{g} \geq (\leq) 0,
    \end{equation}
    then \( \Omega \) is a metric ball, and \( u \) is a radial function.
\end{theoremletter}
In particular, we have the following result for Einstein manifolds:
\begin{coro}\label{corollary-1}
    Let \( (M^n, g) \) be an Einstein manifold with \( \operatorname{Ric}_g = (n-1)k g \), for some \( k \in \mathbb{R} \). Additionally, assume that \( M \) is endowed with a conformal vector field \( X \) whose conformal factor does not change sign. If \( u \) is a solution of \eqref{serrinoka}--\eqref{serrinokb} for a certain domain \( \Omega \), then \( \Omega \) is a metric ball, and \( u \) is a radial function.
\end{coro}
The last results extend known results in several directions. For example, Corollary \ref{corollary-1} is a generalization of Ciraolo and Vezzoni's result \cite{ciraolo} for space forms (manifolds with constant sectional curvature that possess a vector field with such properties). Furthermore, the compatibility condition \eqref{main1} is exactly the same condition observed by Farina and Roncoroni in the specific case of warped products, which also represent a type of manifold that satisfies these conditions (see \cite[Theorem 3]{farina-roncoroni}).

Our method for deriving a suitable Pohozaev-type identity in this context begins with a geometric identity established by Schoen (\cite{schoen}) and involves making an appropriate conformal change that depends on the solution \( u \). This approach allows us to obtain a new Pohozaev-type identity that is interesting in its own right, and we believe it could be applied to other overdetermined problems (see Proposition \ref{poho_type}). We emphasize that, to the best of our knowledge, this method for deriving a Pohozaev identity in the Riemannian context for Serrin's problem is distinct from the other works cited in this field.

Returning to the classical Serrin's problem, we note that the overdetermined condition \( |\nabla u| = c \) on \( \partial \Omega \) is sufficiently strong to prescribe the entire domain. To relax this condition, Agostiniani, Borghini, and Mazzieri (\cite[Section 2]{ABM}) studied the problem for domains \( \Omega \subset \mathbb{R}^2 \) and defined the concept of normalized wall shear stress (NWSS), using it to obtain a rigidity result for the ball in this setting. We will generalize this concept to any dimension and Riemannian manifolds to state our second result. We use the notation \( \pi_{0}(A) \) to denote the collection of connected components of a given set \( A \). Additionally, let \( u_{\max} = \max_{\Omega} u \) and \( \text{Max}(u) = \{ p \in \Omega \mid u(p) = u_{\max} \} \).

\begin{defin}\label{GNWSS}
Let \( (\Omega,u) \) be a solution to the problem \eqref{serrinoka} and let \( \Gamma \) be a connected component of \( \partial \Omega \). We define the \textup{generalized normalized wall shear stress} (GNWSS) of \( \Gamma \) as
\begin{equation}\label{GNWSSboundary}
    \overline{\tau}(\Gamma)=\frac{\max_{\Gamma} \abs{\nabla u}^2}{k u_{\textup{max}}^2+\frac{2}{n}u_{\textup{max}}}.
\end{equation}
More generally, if \( N \in \pi_0 (\Omega \setminus \textup{Max}(u)) \), we define the generalized normalized wall shear stress of \( N \) as 
\begin{equation}\label{GNWSSregion}
    \overline{\tau}(N)= \max \left\{ \overline{\tau}(\Gamma) ~\colon~ \Gamma \in \pi_{0} (\partial \Omega \cap \overline{N}) \right\}.
\end{equation}
\end{defin}
A simple calculation shows that the GNWSS for the classical solution \eqref{class_solution} is precisely \(\bar{\tau}(\Omega\setminus\left\{0\right\}) = 1\), where $\textup{Max}(u)$ is the origin.

With this notation, we can state the following rigidity result for real analytic Riemannian manifolds:
\begin{theoremletter}\label{theorem-B}
 Let $(M,g)$ be a real analytic Riemannian manifold of dimension $n \geq 2$, and let $(\Omega, u)$ be a solution to the problem \eqref{serrinoka}, where $\Omega$ is a bounded domain with $C^2$-boundary in $M$. Let $N \in \pi_{0}(\Omega \setminus \textup{Max}(u))$ such that $\mbox{Ric}_g\geq (n-1)kg$ inside $N$ and $\overline{\tau}(N) \leq 1$. Then, $\Omega$ is a metric ball and $u$ is a radial function.   
\end{theoremletter}
A simple consequence of this result is the generalization of \cite[Theorem 2.1]{ABM} when $(M,g)$ is an Einstein manifold. Note that all Einstein manifolds are real analytic (see \cite{besse}).
\begin{coro}\label{corollary-2}
    Let $(M,g)$ be an Einstein manifold of dimension $n \geq 2$ with $\mbox{Ric}_g = (n-1)kg$, for some \( k \in \mathbb{R} \), and let $(\Omega, u)$ be a solution to the problem \eqref{serrinoka}, where $\Omega$ is a bounded domain with $C^2$-boundary in $M$. Let $N \in \pi_{0}(\Omega \setminus \textup{Max}(u))$ such that $\overline{\tau}(N) \leq 1$. Then, $\Omega$ is a metric ball and $u$ is a radial function.   
\end{coro}
This paper is organized as follows. In Section \ref{sec3}, we fix some notations and prove auxiliary results regarding conformal vector fields. In Section \ref{sec_poho}, we establish a Pohozaev-type identity in Riemannian manifolds endowed with a conformal vector field using the conformal method. Then, we prove Theorem \ref{theorem-A} in Section \ref{sec4}. In Section \ref{sec5}, we consider a positive solution to the Dirichlet problem, define the generalized normalized wall shear stress, and prove Theorem \ref{theorem-B}. Finally, in the Appendix \ref{alternative}, we will give another proof of the Pohozaev-type identity, following a similar approach to \cite{freitas}. We note that our initial proof of the Pohozaev-type formula (Proposition \ref{poho_type}) is valid only in dimensions \(n \geq 3\) due to the proposed conformal change. However, as demonstrated in Appendix A, the formula also holds in dimension \(n = 2\), as shown by our second proof.
\section{Conformal vector fields}\label{sec3}

Let $(M^n, g)$ be a Riemannian manifold of dimension $n \geq 2$. Here, we present some facts about conformal vector fields. 

Recall that a smooth vector field $X$ on $(M^n, g)$ is said to be a conformal vector field if its flow consists of conformal transformations, or equivalently, if there exists a smooth function $\psi$ on $M$ (called the conformal factor of the conformal vector field $X$) such that the Lie derivative of the metric is a multiple of the metric, i.e.,
\[
\mathcal{L}_X g = 2\psi g,
\]
where the Lie derivative $\mathcal{L}_X g$ is given by
\[
\mathcal{L}_X g(Y,Z) = \langle \nabla_Y X, Z \rangle + \langle \nabla_Z X, Y \rangle, \quad X, Y, Z \in \mathfrak{X}(M).
\]
By Koszul's formula, we obtain that
\[
2g(\nabla_Y X, Z) = \mathcal{L}_X g(Y,Z) + dX^b(Y,Z),
\]
where $X^b$ is the $1$-form dual to $X$, defined as $X^b(Z) = g(X,Z)$ for all $Z \in \mathfrak{X}(M)$, and
\[
dX^b(Y,Z) = Y(X^b(Z)) - Z(X^b(Y)) - X^b([Y,Z]).
\]
Since $X$ is conformal, we infer that
\begin{eqnarray}\label{eq001}
2g(\nabla_Y X, Z) = 2\psi g(Y,Z) + dX^b(Y,Z).
\end{eqnarray}
Now, define a skew-symmetric tensor field $\varphi$ of type $(1,1)$ on $M$ by
\begin{equation}\label{eq002}
dX^b(Y,Z) = 2g(\varphi Y, Z), \quad Y, Z \in \mathfrak{X}(M).
\end{equation}
From equations \eqref{eq001} and \eqref{eq002}, we conclude that
\[
\nabla_Y X = \psi Y + \varphi Y.
\]

\begin{remark}
We recall that a conformal vector field $X$ is called \textit{closed conformal} if the 1-form $X^b$ is closed (equivalently, if $\varphi \equiv 0$). This means
\[
\nabla_Y X = \psi Y,
\]
for all $Y \in \mathfrak{X}(M)$. We point out that manifolds endowed with a nontrivial closed conformal vector field possess an interesting property: they are locally isometric to a warped product with a 1-dimensional factor (for more details, see, e.g., \cite[Section 3]{montiel}).
\end{remark}

The following results provide information about the curvature tensor and the gradients of the function $\psi$ and the tensor $\varphi$. In particular, we compute the Ricci tensor.
\begin{lemma}[Lemma 2.1 in \cite{deshmukh}]\label{lemma 2.1} Let $X$ be a conformal vector field on a Riemannian manifold $(M^n, g)$ with conformal factor $\psi$. Then,
$$(\nabla\varphi)(Y,Z)=R_m(Y,X)Z+Z(\psi)Y-g(Y,Z)\nabla\psi,$$
where $\varphi$ is the $(1,1)$-tensor defined in \eqref{eq002}, $(\nabla\varphi)(Y,Z)=\nabla_Y(\varphi Z)-\varphi(\nabla_YZ),$ $R_m$ is the curvature tensor field and $\nabla\psi$ is the gradient of the function $\psi.$ In particular, for any point $p \in M$ and an orthonormal basis $\{e_1, \dots , e_n\} \subset T_p M$, the Ricci tensor in $p$ satisfies
\begin{eqnarray}\label{eq003}
Ric_g |_p (X (p),\cdot )=-(n-1)\nabla \psi (p)-\displaystyle\sum_{i=1}^n\nabla \varphi |_p (e_i,e_i).
\end{eqnarray}
\end{lemma}
From this last result and using Bianchi's identity, we derive the following:
\begin{lemma}\label{lemma1}
Under the conditions of the last lemma, if $\psi$ is the conformal factor of $X$, then
\[
-(n-1)\Delta \psi = \dfrac{1}{2}X(R_g) + \psi R_g,
\]
where $R_g$ is the scalar curvature of $(M^n, g)$.
\end{lemma}
\begin{proof}
Given $p \in M$, consider an orthonormal frame $E_1, \dots, E_n$ in a neighborhood of $p$. Taking the divergence of \eqref{eq003}, and using Bianchi's identity, we obtain
\begin{eqnarray*}
-(n-1)\Delta \psi &=&\text{div} Ric_g(X)+\dfrac{1}{2}\langle Ric_g, \mathcal{L}_Xg\rangle +\text{div}(\displaystyle\sum_{i=1}^n \nabla\varphi(E_i,E_i))\nonumber\\
&=&\dfrac{1}{2}X(R_g)+\psi R_g+\text{div}(\displaystyle\sum_{i=1}^n\nabla\varphi(E_i,E_i)).
\end{eqnarray*}   
To conclude this proof, we just note that
\begin{equation}\label{div_0}
\sum_{i=1}^n \text{div}(\nabla \varphi(E_i, E_i)) = 0.    
\end{equation}
In fact, since $\varphi$ is skew-symmetric, we have, using coordinates in a neighborhood of $p$,
\begin{equation*}\label{vanishingVarPhi} \varphi_{ki,ik}=\varphi_{ki,ki}+\varphi_{ti}R^t_{kik}+\varphi_{kt}R^t_{iik}=-\varphi_{ik,ki}+\varphi_{ti}R^t_{kik}+\varphi_{kt}R^t_{iik}.
\end{equation*}
This implies that
\[
2\varphi_{ki,ik} = \varphi_{ti} R^t_i - \varphi_{kt} R^t_k = 2\varphi_{ti} R^t_i = 0,
\]
where the last equality holds because $\varphi$ is skew-symmetric and the Ricci tensor $Ric$ is symmetric.  
\end{proof}

\section{A Pohozaev-type identity }\label{sec_poho}
\indent\par
In this section, we consider a Riemannian manifold $(M^n, g)$, where $n \geq 3$, endowed with a conformal vector field $X$ with a conformal factor $\psi$. The purpose here is to obtain a Pohozaev-type identity using the conformal method, starting with a geometric identity due to Schoen. The proof of this geometric identity is straightforward and is accomplished using integration by parts and the contracted second Bianchi identity.

We recall that Dirichlet solutions of certain semilinear elliptic equation in domains of Euclidean space satisfies an integral identity, which is known as  Pohozaev identity. Bourguignon and Ezin \cite{bourguignon1987scalar} proved that if we consider a closed Riemannian manifold with a conformal vector field, then the scalar curvature satisfies a similar integral identity. Afterwards, Schoen, see \cite{schoen}, generalized this identity.
\begin{theorem}[Pohozaev-Schoen Identity \cite{schoen}]
 Let $(M^{n},\bar{g})$ be a Riemannian manifold, $\Omega\subset M$ a bounded domain possibly with a boundary $\partial\Omega$, and $Y$ a vector field in $M$. Then,
\begin{eqnarray}\label{Pohid}
\displaystyle\int_{\Omega} Y(R_{\bar{g}})dV_{\bar{g}}=-\frac{n}{n-2}\displaystyle\int_{\Omega}\bar{g}\left(\stackrel{\bar{\circ}}{\Ricc_{\bar{g}}},\mathcal{L}_{Y}\bar{g}\right)dV_{\bar{g}}+\frac{2n}{n-2}\displaystyle\int_{\partial\Omega}\stackrel{\bar{\circ}}{\Ricc_{\bar{g}}}\left(Y,\bar{\nu}\right)d\sigma_{\bar{g}}, \label{pschoen} 
\end{eqnarray}
where $\stackrel{\bar{\circ}}{\Ricc_{\bar{g}}}$ is the traceless Ricci tensor on the metric $\bar{g}$.
\end{theorem}

\begin{remark}
There is an intriguing link between the Pohozaev-Schoen Identity mentioned above and the classical Pohozaev Identity \cite{pohozaev}. A proof of the latter can be found, for example, in \cite[page~171]{struwe}. This relationship was noted in \cite[Section 5]{rigoli} (see also \cite[Section 4.8]{gover}). Starting from \eqref{pschoen} for $(\Omega, \bar{g})$, where $\Omega \subset \mathbb{R}^{n}$ is a domain and $\bar{g} = u^{\frac{4}{n-2}}g_{0}$ with $g_{0}$ being the Euclidean metric, the idea is to consider the position vector in Euclidean space and rewrite \eqref{pschoen} in terms of the metric $g_{0}$. This served as a motivating point to explore a suitable conformal change in our setting.\end{remark}

Given the relevance of \eqref{Pohid} in connection with important problems in Geometric Analysis, it is natural to use this identity and the conformal geometry of the manifold to obtain some rigidity results in the context of Serrin's type problem. We consider \( u \) as a solution for the Serrin's-type problem.
\begin{eqnarray} \label{serrinproblem}
\left\{
\begin{array}{rcl}
     \Delta u +nku& = & -1 ,\ \text{on}\  \Omega \\
     u& > & 0,\ \text{in}\  \operatorname{int}(\Omega) \\
u & = &0, \ \text{on}\ \partial \Omega\\
|\nabla u| &=& c, \ \text{on } \partial \Omega.
\end{array}
\right.
\end{eqnarray}
In this way, we consider the conformal change $$\bar{g}=u^{\frac{4}{n-2}}g=e^{2\beta}g,$$
where $u$ satisfies \eqref{serrinproblem}. This is a smooth metric in the set $\Omega_{\varepsilon}=\Omega\setminus u^{-1}([0,\varepsilon)),$ where $\varepsilon>0$. 
The field $X$ remains conformal in $\bar{g},$ because
\begin{eqnarray*}   \mathcal{L}_{X}\bar{g}&=&\mathcal{L}_{X}(e^{2\beta}g)\\    &=&e^{2\beta}\mathcal{L}_{X}g+2e^{2\beta}X(\beta)g\\
&=&2(\psi+2X(\beta))\bar{g}.
\end{eqnarray*}
By considering $Y=X$, which is a conformal vector field in $(M^{n},\bar{g})$, \eqref{pschoen} becomes
\begin{equation}
\displaystyle\int_{\Omega_{\varepsilon}} X(R_{\bar{g}})dV_{\bar{g}}=\frac{2n}{n-2}\displaystyle\int_{\partial\Omega_{\varepsilon}}\stackrel{\bar{\circ}}{\Ricc_{\bar{g}}}\left(X,\bar{\nu}\right)d\sigma_{\bar{g}}. \label{pschoen2} 
\end{equation}
Now, we rewrite \eqref{pschoen2} in terms of the metric $g$. First,
\begin{eqnarray*}
\Ricc_{\bar{g}}&=&\Ricc_{g}-(n-2)(\nabla^2\beta-d\beta\otimes d\beta)-[\Delta_{g}\beta+(n-2)|\nabla\beta|^{2}]g\\
 &=&\Ricc_{g}-(n-2)\left[\frac{2}{n-2}\left(\frac{\nabla^2 u}{u}-\frac{1}{u^{2}}du\otimes du\right)-\frac{4}{u^{2}(n-2)^{2}}du\otimes du\right]\\
 & &-\left[\frac{2}{n-2}\left(\frac{\Delta u}{u}-\frac{|\nabla u|^{2}}{u^{2}}\right)+\frac{4(n-2)}{u^{2}(n-2)^{2}}|\nabla u|^{2}\right]g,
\end{eqnarray*}
this implies
\begin{equation*}
 \Ricc_{\bar{g}}=\Ricc_{g}-\frac{2}{u}\nabla^2 u+\frac{2n}{u^{2}(n-2)}du\otimes du-\frac{2}{u(n-2)}(\Delta u)g-\frac{2}{u^{2}(n-2)}|\nabla u|^{2}g.
\end{equation*}
Second, by tracing the above expression, we obtain
\begin{equation}
R_{\bar{g}}=u^{-\frac{4}{n-2}}\left[R_{g}-4\left(\frac{n-1}{n-2}\right)\frac{\Delta u}{u}\right].\label{scalarbarg}
\end{equation}
And, then,
\begin{eqnarray}
\stackrel{\bar{\circ}}{\Ricc_{\bar{g}}}=\stackrel{\circ}{\Ricc_{g}}-\frac{2}{u}\nabla^2 u+\frac{2n}{u^{2}(n-2)} du\otimes du+\frac{2}{un}(\Delta u)g-\frac{2}{u^{2}(n-2)}|\nabla u|^{2}g.  \label{ricci0barg}   
\end{eqnarray}
Therefore,
\begin{eqnarray*}
X(R_{\bar{g}})&=&X\left(u^{-\frac{4}{n-2}}R_{g}-4\left(\frac{n-1}{n-2}\right)u^{-\frac{n+2}{n-2}}\Delta u\right)\\
 & = &u^{\frac{-4}{n-2}}X(R_{g})+R_{g}\left(-\frac{4}{n-2}u^{-\frac{n+2}{n-2}}g(\nabla u,X)\right)+\frac{4(n-1)(n+2)}{(n-2)^{2}}u^{-\frac{2n}{n-2}}g(\nabla u,X)\Delta u\\
& & -\frac{4(n-1)}{n-2}u^{-\frac{n+2}{n-2}}g(X,\nabla\Delta u). 
\end{eqnarray*}
Since $dV_{\bar{g}}=u^\frac{2n}{n-2}dV_{g}$, it follows
\begin{eqnarray}
 \displaystyle\int_{\Omega_{\varepsilon}} X(R_{\bar{g}})dV_{\bar{g}}&=&\displaystyle\int_{\Omega_{\varepsilon}}u^{2}X(R_{g})dV_{g} -\frac{4}{n-2}\displaystyle\int_{\Omega_{\varepsilon}}uR_{g} g(\nabla u,X)dV_{g} \nonumber\\
 & &+\frac{4(n-1)(n+2)}{(n-2)^{2}}\displaystyle\int_{\Omega_{\varepsilon}}g(\nabla u, X)\Delta u dV_{g}-\frac{4(n-1)}{n-2}\displaystyle\int_{\Omega_{\varepsilon}}u g(X,\nabla\Delta u)dV_{g}. \label{lhspschoen}
\end{eqnarray}
Using that $\Delta u =-1-nku$ inside $\Omega$, we calculate the last two terms of \eqref{lhspschoen}. More precisely, 
\begin{eqnarray*}
&&\frac{4(n-1)(n+2)}{(n-2)^{2}}\displaystyle\int_{\Omega_{\varepsilon}}g(\nabla u, X)\Delta u dV_{g}-\frac{4(n-1)}{n-2}\displaystyle\int_{\Omega_{\varepsilon}}u g(X,\nabla\Delta u)dV_{g}\\
 &=&\frac{4(n-1)}{n-2}\left[\frac{n+2}{n-2}\displaystyle\int_{\Omega_{\varepsilon}}g(\nabla u, X)(-1-nku)dV_{g}+\displaystyle\int_{\Omega_{\varepsilon}}knu g(X,\nabla u)dV_{g}\right]\\
 &=&\frac{4(n-1)}{n-2}\left[-\left(\frac{n+2}{n-2}\right)\displaystyle\int_{\Omega_{\varepsilon}}g(\nabla u, X)dV_{g}-\frac{4nk}{n-2}\displaystyle\int_{\Omega_{\varepsilon}}ug(X,\nabla u)dV_{g}\right]\\
 & = & \frac{4n(n-1)(n+2)}{(n-2)^{2}}\displaystyle\int_{\Omega_{\varepsilon}} u\psi dV_{g}+\frac{8(n-1)n^{2}k}{(n-2)^{2}}\displaystyle\int_{\Omega_{\varepsilon}}u^{2}\psi dV_{g}+\xi_1(\varepsilon),
\end{eqnarray*}
where $\displaystyle\lim_{\varepsilon\to 0}\xi_{1}(\varepsilon)=0.$
Third equality above follows from the divergence theorem, along with the fact that $X$ is a conformal vector field with conformal factor $\psi$, which yields the following expressions
\begin{eqnarray*}
\displaystyle\int_{\Omega_{\varepsilon}}ug(X,\nabla u)dV_{g}&=&-\frac{1}{2}\displaystyle\int_{\Omega_{\varepsilon}}u^{2}\mbox{div} XdV_{g}+\frac{1}{2}\varepsilon^{2}\int_{\partial\Omega_{\varepsilon}}\langle X,\nu\rangle d\sigma_{g}\\
&=&-\frac{1}{2}\displaystyle\int_{\Omega_{\varepsilon}}nu^{2}\psi dV_{g}+\frac{1}{2}\varepsilon^{2}\int_{\partial\Omega_{\varepsilon}}\langle X,\nu\rangle d\sigma_{g}
\end{eqnarray*}
and
\begin{eqnarray*}
\displaystyle\int_{\Omega_{\varepsilon}} g(\nabla u,X)dV_{g}&=&-\displaystyle\int_{\Omega_{\varepsilon}}u\mbox{div} X dV_{g}+\varepsilon\int_{\partial\Omega_{\varepsilon}}\langle X,\nu\rangle d\sigma_{g}\\
&=&-n\displaystyle\int_{\Omega_{\varepsilon}}u\psi dV_{g}+\varepsilon\int_{\partial\Omega_{\varepsilon}}\langle X,\nu\rangle d\sigma_{g}. 
\end{eqnarray*}
Similarly, since
\begin{eqnarray*}
 \mbox{div}(u^{2}R_{g}X)&=&u^{2}R_{g}\mbox{div}X+u^{2}X(R)+2R_{g}ug(\nabla u, X)\\
 &=& nu^{2}R_{g}\psi+u^{2}X(R_{g})+2R_{g}ug(\nabla u, X),
\end{eqnarray*}
we have
\begin{eqnarray*}
\displaystyle\int_{\Omega_{\varepsilon}}R_{g}u g(\nabla u, X)dV_{g}=-\frac{1}{2}\displaystyle\int_{\Omega_{\varepsilon}}u^{2} X(R_{g}) dV_{g}-\frac{n}{2}\displaystyle\int_{\Omega_{\varepsilon}}u^{2}\psi R_{g}dV_{g}+\xi_{2}(\varepsilon),
\end{eqnarray*}
with $\displaystyle\lim_{\varepsilon\to 0}\xi_{2}(\varepsilon)=0$.
Replacing these results in \eqref{lhspschoen}, we obtain, after doing $\varepsilon \to 0$,
\begin{eqnarray}
\displaystyle\int_{\Omega}X(R_{\bar{g}})dV_{\bar{g}}
=\frac{2n}{n-2}\displaystyle\int_{\Omega}\left[u^{2}\left(\frac{X(R_{g})}{2}+\psi R_g\right)+\frac{2(n-1)(n+2)}{n-2}u\psi +\frac{4n(n-1)k}{n-2}u^{2}\psi\right]dV_{g}.\label{lhspschoen2}
\end{eqnarray}
On other hand, since $\bar{\nu}=u^{-\frac{2}{n-2}}\nu$ and $d\sigma_{\bar{g}}=u^{\frac{2(n-1)}{n-2}}d\sigma_{g}$, we use \eqref{ricci0barg} to write
\begin{eqnarray*}
\displaystyle\int_{\partial\Omega_{\varepsilon}}\stackrel{\bar{\circ}}{\Ricc_{\bar{g}}}\left(X,\bar{\nu}\right)d\sigma_{\bar{g}}&=&\displaystyle\int_{\partial\Omega_{\varepsilon}}u^{2}\stackrel{\bar{\circ}}{\Ricc_{\bar{g}}}\left(X,\nu\right)d\sigma_{g}\\
&=&\displaystyle\int_{\partial\Omega_\varepsilon}\left[u^{2}\stackrel{\circ}{\Ricc_{g}}(X,\nu)-2u\nabla^2 u(X,\nu)+\frac{2n}{n-2} (du\otimes du)(X,\nu)\right]d\sigma_{g}\\
& &+\displaystyle\int_{\partial\Omega_\varepsilon}\left[\frac{2u}{n}(\Delta u)g(X,\nu)-\frac{2}{(n-2)}|\nabla u|^{2}g(X,\nu)\right]d\sigma_{g}\\
&=&\frac{2}{n-2}\left[\displaystyle\int_{\partial\Omega_\varepsilon}(ng(\nabla u,X)g(\nabla u,\nu)-|\nabla u|^{2}g(X,\nu))d\sigma_{g}\right]+\zeta(\varepsilon)\\
&=& \frac{2(n-1)}{n-2}\displaystyle\int_{\partial\Omega_\varepsilon}|\nabla u|^{2}g(X,\nu)d\sigma_{g}+\zeta(\varepsilon)\\
&=&\frac{2(n-1)}{n-2}c^{2}\displaystyle\int_{\partial\Omega_\varepsilon}g(X,\nu)d\sigma_{g}+\zeta(\varepsilon),
\end{eqnarray*}
where $\displaystyle\lim_{\varepsilon\to 0}\zeta(\varepsilon)=0$. In the fourth equality we have used that $\nu=-\frac{\nabla u}{|\nabla u|}$, and in the fifth we have used that $|\nabla u|=c$ on $\partial\Omega$. By doing $\varepsilon\to 0$ in the above identity and using again the Divergence Theorem and that $\mbox{div}X=n\psi$, we conclude
\begin{equation}
\displaystyle\int_{\partial\Omega}\stackrel{\bar{\circ}}{\Ricc_{\bar{g}}}\left(X,\bar{\nu}\right)d\sigma_{\bar{g}}=\frac{2n(n-1)}{n-2}c^{2}\int_{\Omega}\psi dV_{g}.\label{ricci0barg2}  
\end{equation}
By replacing \eqref{lhspschoen2} and \eqref{ricci0barg2} in \eqref{pschoen2}, we conclude the following Pohozaev-type
identity in our underlying analysis
\begin{proposition}[Pohozaev-type identity]\label{poho_type}
Let $(M^n,g)$ be a Riemannian manifold of dimension $n\geq 3$ endowed with a conformal vector field $X$ with conformal factor $\psi$. If $\Omega$ is a bounded domain and $u$ is a solution for \eqref{serrinoka}, then
 \begin{eqnarray}
\frac{n-2}{2n(n-1)}\displaystyle\int_{\Omega}u^{2}\left(\frac{X(R_{g})}{2}+\psi R_g\right)dV_{g}+\frac{n+2}{n}\displaystyle\int_{\Omega}u\psi  dV_{g}+2k\displaystyle\int_{\Omega}u^{2}\psi dV_{g}=c^{2}\int_{\Omega}\psi dV_{g}. \label{conformal}
\end{eqnarray}   
\end{proposition}
\begin{remark}
The identity \eqref{conformal} is the same as that obtained by the second named author with Roncoroni and Santos (see \cite[Lemma 3]{freitas}) in the case where \(X\) is closed conformal. We highlight that the method used here is quite different. For a related approach to the previous work, see Appendix \ref{alternative}.
\end{remark}

\section{A rigidity result for Serrin's problem}\label{sec4}
In this section, we prove a rigidity result for Serrin's problem in Riemannian manifolds endowed with a conformal vector field. To do this, we use the Pohozaev-type identity associated with the solution of \eqref{serrinproblem} to apply the \(P\)-function approach.

Farina and Roncoroni \cite{farina-roncoroni} proved an Obata-type rigidity result with the hypothesis that the Ricci curvature is bounded from below. In fact, they showed that under this condition, if \( u \in C^2(\Omega) \) is a solution to 
\begin{eqnarray}\label{eqlapla}
\Delta u + nku = -1, \quad \text{in } \Omega,
\end{eqnarray}
where \( \Omega \subset M \) is a domain, then the \( P \)-function defined in \eqref{pfunction} is subharmonic. Moreover, they showed that the harmonicity of \( P \) implies that 
\[
Ric_{g}(\nabla u, \nabla u) = (n-1)k|\nabla u|^2 \quad \text{in } \Omega,
\]
and that \( u \) satisfies the equation
\[
\nabla^2 u = - \left( \frac{1}{n} + ku \right) g \quad \text{in } \Omega.
\]
Thus, an Obata-type result implies that the rigidity of a solution to \eqref{serrinproblem} follows from the harmonicity of the \( P \)-function \eqref{pfunction}. More precisely, they proved the following rigidity result:
\begin{lemma}[\cite{farina-roncoroni}]\label{farina-roncoroni}
Let $(M^{n},g)$ be an $n$-dimensional Riemannian manifold such that  
\begin{equation}\label{ricci_bound}
\mbox{Ric}_g\geq (n-1)kg, \text{for } k\in\mathbb{R}.    
\end{equation}
Let $\Omega\subset M$ be a domain and $u\in C^{2}(\Omega)$ a solution of \eqref{eqlapla} in $\Omega$. Then 
$$\Delta P(u)\geq 0.$$
Furthermore, $\Delta P(u)=0$ if and only if 
$\Omega$ is a metric ball and $u$ is a radial function, i.e. u depends only on the distance from the center of the ball.
\end{lemma}

\begin{remark}
In the last result, the necessary condition is only \(\mbox{Ric}_g(\nabla u, \nabla u) \geq (n-1)k|\nabla u|^2\).
\end{remark}
In particular, if \eqref{ricci_bound} and \(|\nabla u| = c \text{ on } \partial \Omega\), we can apply the maximum principle to the subharmonic function \(P(u)\). This gives us:

\begin{itemize}
\item[(a)] either \(P(u) = c^2\) in \(\overline{\Omega}\); 
\item[(b)] or \(P(u) < c^2\) in \(\Omega\).
\end{itemize}

Following Weinberger's method, the idea is to look for a geometric condition such that (b) cannot occur; this would imply that \(P(u)\) is constant, and the rigidity follows from Lemma \ref{farina-roncoroni}. With these tools, we are able to demonstrate the main result of this section.
\begin{proof}[Proof of Theorem \ref{theorem-A}]
Since Proposition \ref{poho_type} gives us the same identity as in \cite[Lemma 3]{freitas}, the proof follows exactly as in the proof of \cite[Theorem 5]{freitas}, but we write it here for the sake of completeness. Let us suppose that the conformal factor \(\psi\) is positive (the argument with \(\psi < 0\) is the same) and that (b) occurs. Then, multiplying \eqref{pfunction} by \(\psi\) and integrating in \(\Omega\), we have
\begin{equation}\label{eq0}
c^2\displaystyle\int_{\Omega}\psi dV_{g}>\displaystyle\int_{\Omega}\psi|\nabla u|^2 dV_{g}+\dfrac{2}{n}\displaystyle\int_{\Omega}\psi u dV_{g}+k\displaystyle\int_{\Omega}\psi u^2 dV_{g}.
\end{equation}
Since
$\text{div}(u^2\nabla \psi)=u^2\Delta\psi+2u\langle \nabla u,\nabla \psi\rangle,$
and 
$\text{div}(\psi u\nabla u)=u\langle \nabla u,\nabla \psi\rangle+\psi|\nabla u|^2+\psi u\Delta u,$ we obtain that $\text{div}(u^2\nabla\psi)=u^2\Delta\psi+2\text{div}(\psi\nabla u)-2\psi|\nabla u|^2-2\psi\Delta u$. Now, using that $u=0$ on $\partial\Omega$, the Divergence Theorem and Lemma \ref{lemma1}, we conclude
\begin{eqnarray}\label{eq015}
\displaystyle\int_{\Omega}\psi|\nabla u|^2 dV_{g}&=&\dfrac{1}{2}\displaystyle\int_{\Omega}u^2\Delta \psi dV_{g}-\displaystyle\int_{\Omega}\psi u\Delta u dV_{g}\nonumber \\
&=&-\dfrac{1}{2(n-1)}\displaystyle\int_{\Omega}u^2\left[\dfrac{X(R_g)}{2}+\psi R_g\right]dV_{g}+\displaystyle\int_{\Omega}\psi u(1+nku)dV_{g}.\label{eq2}
\end{eqnarray}
Replacing \eqref{eq2} in \eqref{eq0} and using the identity obtained in Proposition \ref{poho_type}, we can rearrange the terms to arrive at
$$\displaystyle\int_{\Omega} u^2\left(-\dfrac{X(R_g)}{2}+\psi(-R_g+n(n-1)k\right)dV_{g}<0,$$
precisely a contradiction if we initially suppose the compatibility condition \eqref{main1}. Therefore, $P(u)$ is constant and, in particular, is harmonic. The rigidity follows from Lemma \ref{farina-roncoroni}.

\end{proof}
As an application, we approach the case where the manifold is Einstein.
\begin{proof}[Proof of Corollary \ref{corollary-1}]
Since \( R_g = n(n-1)k \) is constant, this implies that \( X(R_g) = 0 \). Consequently, the conditions \eqref{ricci_bound} and \eqref{main1} of Theorem \ref{theorem-A} are satisfied.
\end{proof}

\begin{remark}
It is a well-known fact that complete Einstein manifolds endowed with a conformal vector field are well classified. In \cite[Theorem 3.1]{kr}, the authors proved that if the conformal factor is non-constant, then \( M \) is a warped product. Alternatively, the last result could be applied in complete Einstein manifolds to guarantee that, when it admits a solution, this manifold is a warped product. In fact, the rigidity for Serrin's problem in this particular case also implies that \( \mathring{\nabla}^2 u = 0 \) in \( \Omega \), and by using Tashiro's classical result \cite{tashiro} (see also \cite[Section 1]{cheeger}), we conclude that \( \Omega \) is a warped product. Since \((M^n, g)\) is Einstein, by analyticity (see \cite[Theorem 5.26]{besse}), we conclude that the entire manifold is a warped product. 
\end{remark}

\begin{remark}
It follows from \cite[Corollary 2.2]{kr} that an Einstein manifold \((M^n, g)\) equipped with a homothetic conformal vector field (which means that the conformal factor is constant) must be Ricci flat. In this case, a solution to \eqref{serrinproblem} is constructed in \cite{fall}, provided that the manifold \(M\) is compact. It follows then, from the previous Remark, that the Riemannian manifold \((M^n, g)\) must be a warped product.
\end{remark}


\section{The Generalized Normalized Wall Shear Stress and a rigidity result for Serrin's problem} \label{sec5}
In this section, we consider a real analytic Riemannian manifold \((M^n, g)\) of dimension \(n \geq 2\) such that \(\text{Ric}_g \geq (n-1)kg\), where \(k \in \mathbb{R}\). Let \((\Omega, u)\) be a positive solution to the Dirichlet problem given by 
\begin{eqnarray}\label{serrinok3}
\Delta u + nku &=& -1, \quad u > 0 \quad \text{in } \Omega, \quad u = 0 \quad \text{on } \partial \Omega,
\end{eqnarray}
where \(\Omega \subset M^{n}\) is a bounded domain with a \(C^2\)-boundary. We denote the maximum value of \(u\) inside \(\Omega\) as \(u_{\text{max}} = \max_{\overline{\Omega}} u\), and its corresponding level set as \(\text{Max}(u) = u^{-1}(u_{\text{max}}) \cap \Omega\).

In \cite{ABM}, the problem \eqref{serrinok3} was studied in the case where \((M, g)\) is the 2-dimensional Euclidean space. The authors proved several interesting results, including some rigidity results for the solutions of \eqref{serrinok3} without the need to impose constant Neumann data. Moreover, it suffices to impose a bound on a normalization of \(\abs{\nabla u}\) along the boundary of \(\Omega\) (see \cite[Definition 1.3]{ABM}). Motivated by this work, we will show that the results in \cite[Section 2]{ABM} are fully generalizable to the case of Serrin's problem over real analytic Riemannian manifolds with a bound on the Ricci curvature. In the realm of static metrics, this quantity was approached by \cite{BM}. 

First, we start with the following:

\begin{definition}
Let \((\Omega, u)\) be a solution to the problem \eqref{serrinok3} and let \(\Gamma\) be a connected component of \(\partial \Omega\). We define the \textup{generalized normalized wall shear stress} of \(\Gamma\) as
\begin{equation}
    \overline{\tau}(\Gamma) = \frac{\max_{\Gamma} \abs{\nabla u}^2}{k u_{\text{max}}^2 + \frac{2}{n} u_{\text{max}}}.
\end{equation}
More generally, if \(N\) is a connected component of \(\Omega \setminus \text{Max}(u)\), we define the \textup{generalized normalized wall shear stress} of \(N\) as 
\begin{equation}
    \overline{\tau}(N) = \max \left\{ \overline{\tau}(\Gamma) \colon \Gamma \in \pi_{0} (\partial \Omega \cap \overline{N}) \right\},
\end{equation}
where we say that \(\Gamma \in \pi_{0} (\partial \Omega \cap \overline{N})\) if \(\Gamma\) is a connected component of \(\partial \Omega \cap \overline{N}\).
\end{definition}
\begin{remark}\label{remCrit}
We note that if \(N \subset \Omega \setminus \text{Max}(u)\) and the function \(u\) is nonconstant, then \(\overline{N} \cap \partial \Omega \neq \emptyset\). In fact, otherwise this would imply that \(u\) has an interior minimum inside \(N\), which contradicts the maximum principle. For more details, see Lemma 5.1 in \cite{BM}.
\end{remark}
Observe that this definition reduces to the one introduced in \cite{ABM} when \(n=2\) and \(k=0\). A variation of this definition was also used in \cite{EMa} to classify solutions \((\Omega, u)\) to the eigenvalue problem \(\Delta u + 2u=0\) on the 2-dimensional sphere \(\mathbb{S}^2\) in terms of its \textit{normalized wall shear stress}. For this, they used an appropriate \(P\)-function. Inspired by these results, we observe that the usefulness of the above definition relies on its relation with the \(P\)-function \eqref{pfunction}. In this way, we obtain 
\begin{lemma}\label{lemmaDisk}
Let \((\Omega, u)\) be a solution to the problem \eqref{serrinok3}, where \(\Omega\) is a bounded domain with a smooth boundary in a Riemannian manifold \((M^{n}, g)\) such that \(\text{Ric}_g \geq (n-1)k g\), \(k \in \mathbb{R}\). If 
\begin{equation}\label{limitationNabla}
\max_{\partial\Omega}\frac{|\nabla u|^{2}}{ku^{2}_{\max}+\frac{2}{n}u_{\max}} \leq 1,
\end{equation}
then \(\Omega\) is a metric ball in \(M\) and \(u\) is a radial function.
\end{lemma}

\begin{proof}
By Lemma \ref{farina-roncoroni}, we know that \(P(u)\) is a subharmonic function. Therefore, applying the maximum principle, we obtain
\[
\max_{\Omega} P(u) = \max_{\partial\Omega} P(u) \leq \max_{\partial\Omega} |\nabla u|^{2} \leq ku^{2}_{\textup{max}} + \frac{2}{n} u_{\max}.
\]

Let \(x\) be a maximum point of \(u\) in \(\Omega\). Then, we have \(\nabla u(x) = 0\), which implies

\[
P(u(x)) = \frac{2}{n} u_{\max} + k u^{2}_{\max}.
\]

Since \(P(u)\) attains its maximum at an interior point, it follows that \(P(u)\) is constant, and consequently, \(\Delta P(u) = 0\). Thus, the rigidity statement follows again from Lemma \ref{farina-roncoroni}.
\end{proof}
Note that if $\textup{Max} (u)$ consist only in isolated points, then the left hand side of \eqref{limitationNabla} is equal to $\overline{\tau} (\Omega \setminus \textup{Max}(u))$. The main result of this section consists to show that we can relax hypothesis \eqref{limitationNabla} in Lemma \ref{lemmaDisk} to get the rigidity of $(\Omega,u)$. Indeed, it is enough to have the existence of a connected subset $N \subset \Omega \setminus \textup{Max}(u)$ with $\overline{\tau} (N) \leq 1$. To prove this result, we introduce an \textit{energy function}. This kind of function was introduced in \cite{AM} in the study of solutions to the \textit{Vacuum Einstein Field Equations}, and then was adapted to the realm of overdetermined problems in \cite{ABM}. In this context, it was also used in \cite{EMa}.

Fix a connected component $N \subset \Omega \setminus$Max$(u)$ and let $\textup{Reg}^N (u)$ be the set of regular values of the function $u$ restricted to $N$. Since a solution to \eqref{serrinok3} is analytic, the critical values of $u$ are isolated (see \cite{Souc}) and then $\textup{Reg}^N (u) \subset (0, u_{\textup{max}})$ consists of a finite union of intervals. We consider the function $U: \textup{Reg}^N (u) \to \mathbb{R}$ defined by
\begin{equation}\label{energyFunction}
    U(t)= \frac{1}{\beta (t)^{\frac{n}{2}}} \int_{\{u=t\} \cap \overline{N}} \abs{\nabla u} d \sigma_g, 
\end{equation}
where
\begin{equation}\label{betaFunction}
    \beta (t):= \left( u_{\textup{max}}-t \right)+\frac{n k}{2}\left( u_{\textup{max}}^2-t^2 \right).
\end{equation}
This function is continuous in $\textup{Reg}^N (u)$, but it turns out that it is also non-increasing provided that $(\Omega, u)$ satisfies some conditions. In order to prove this claim, we will need the following Lemma in the case of negative curvature.
\begin{lemma}\label{lemmaMax}
    Let $(\Omega, u)$ be a solution to problem \eqref{serrinok3}, where $\Omega$ is a bounded domain with smooth boundary in a Riemannian manifold $(M^{n}, g)$. Let us suppose that $k<0$. Then it must be $$u_{\textup{max}}< -\frac{1}{nk}.$$
\end{lemma}
\begin{proof}
   Consider $k=-\kappa$ for some $\kappa >0$. Suppose, by contradiction, that $u_{\textup{max}} \geq 1/n\kappa$. Since $u=0$ along $\partial \Omega$ and $u>0$ inside $\Omega$, it follows that $u$ reach its maximum value inside $\Omega$. Define the constant function 
   \[
   v = \frac{1}{\kappa n} \quad \textup{inside } \Omega.
   \]
   Then, $v$ solves the equation $\Delta v - n \kappa v=-1$. Thus, the function $w=u-v$ defined in $\Omega$ is a solution to the equation
   \[
   \Delta w -n\kappa w=0 \quad \textup{in } \Omega.
   \]
   But, this is a contradiction with the weak maximum principle, because $w$ has a non-negative maximum inside $\Omega$. Thus, the Lemma is proved.
\end{proof}
Now, in the next result we prove the monotonicity of the energy function $U$. The limits are always taken in regular values.
\begin{proposition}\label{propEnergy}
Let $(\Omega, u)$ be a solution to problem \eqref{serrinok3}, where $\Omega$ is a bounded domain with $\mathcal{C}^2$-boundary in a real analytic Riemannian manifold $(M^{n}, g)$. Let $N \subset \Omega \setminus \textup{Max} (u)$ be a connected component such that $\mbox{Ric}_g\geq (n-1)kg$ inside $N$ and $\overline{\tau}(N) \leq 1$. Then, the function $U$ is non-increasing.
\end{proposition}
\begin{proof}
    It follows as in the proof of Proposition 2.3 of \cite{ABM}. Let $\epsilon>0$ be such that $u_{\text{max}}-\epsilon \in \textup{Reg}^N (u)$, so the inner domain
    \begin{equation}\label{innerDomain}
        N_{\epsilon} = \left\{ p \in N ~\colon~ u(p)< u_{\text{max}}-\epsilon\right\}
    \end{equation}
    has regular boundary. By Lemma \ref{farina-roncoroni} we have that the $P$-function $P (u)$ defined in \eqref{pfunction} is subharmonic, so applying the maximum principle, we have 
    \[
\max_{N_\epsilon} P(u) \leq \max_{\partial N_\epsilon} P(u).
    \]
    On the other hand, we notice that
    \[
\lim_{\epsilon \to 0^+} \max_{\partial N_\epsilon} P(u) \leq  \frac{2}{n} u_{\text{max}}+k u_{\text{max}}^2,   
    \]
  because $P(u) = \abs{\nabla u}^2 \leq \frac{2}{n} u_{\text{max}}+k u_{\text{max}}^2$ in $\overline{N} \cap \partial \Omega$ (since $\overline{\tau} (N) \leq 1$ by hypothesis) and
    \[
    \lim_{\epsilon \to 0^+} \max_{\partial N_\epsilon \cap \textup{Max}(u)} P(u) = \frac{2}{n} u_{\text{max}}+k u_{\text{max}}^2.
    \]
     It follows that
    \begin{equation}\label{boundP}
        P(u) \leq \frac{2}{n} u_{\text{max}}+k u_{\text{max}}^2 \quad \text{inside } N.
    \end{equation}
    Using this fact, \eqref{serrinok3} and \eqref{betaFunction}, we get 
\begin{eqnarray*}
 \text{div}  \left(\frac{\nabla u}{\beta(u)^{\frac{n}{2}}}\right)&=&\frac{1+n k u}{\beta(u)^{\frac{n}{2}+1}} \left( \frac{n}{2} \abs{\nabla u}^2 -\beta(u) \right)\\
&=&\frac{n}{2}\frac{(1+nku)}{\beta(u)^{\frac{n}{2}+1}} \left( P(u)-\frac{2}{n} u_{\textup{max}}-k u_{\textup{max}}^2 \right) \leq 0,
\end{eqnarray*}
%
%
inside $N$. Observe that, in the case of negative curvature, we have used Lemma \ref{lemmaMax} to obtain  $1+nku >0$ in $N$. Now choose $t_1, t_2 \in \textup{Reg}^N (u)$ with $t_1 < t_2$, such that 
\[
V=\left\{ p \in N ~\colon~ t_1< u(p)<t_2 \right\} \subset N
\]
is an open set with analytic boundary. Let $\nu$ be the outer unit normal to $\partial V$. Then, if we define $\beta (t)$ as in \eqref{betaFunction}, we get the desired result applying the Divergence Theorem. In fact,
\begin{equation*}
    \begin{split}
        U(t_2)-U(t_1) &= \frac{1}{\beta (t_2)^{\frac{n}{2}}} \int_{\{u=t_2\} \cap \overline{N}} \abs{\nabla u} d \sigma_g- \frac{1}{\beta (t_1)^{\frac{n}{2}}} \int_{\{u=t_1\} \cap \overline{N}} \abs{\nabla u} d \sigma_g \\
        &=  \int _{\{u=t_2\}} \left\langle \frac{\nabla u}{\beta (t_2)^{\frac{n}{2}}},\nu\right\rangle d\sigma_g + \int _{\{u=t_1\}} \left\langle \frac{\nabla u }{\beta (t_1)^{\frac{n}{2}}},\nu \right\rangle d\sigma_g \\
        &= \int_{V} \text{div} \left( \frac{\nabla u}{\left( \left( u_{\text{max}}-u \right)+\frac{n k}{2}\left( u_{\text{max}}^2-u^2 \right) \right)^{\frac{n}{2}}} \right) dV_g \leq 0,
    \end{split}
\end{equation*}
where we have used that $\nu =- \nabla u / \abs{\nabla u}$ along $\{u=t_1\}$ and $\nu =\nabla u / \abs{\nabla u}$ along $\{u=t_2\}$, since $t_1$ and $t_2$ are regular values of $u$ and $\nabla u$ points in the direction of maximum increase.
\end{proof}
\begin{proof} [Proof of Theorem \ref{theorem-B}]
    We will prove that it must be \(\overline{N} = \overline{\Omega}\), so the statement will follow from Lemma \ref{lemmaDisk}. To do so, we will show that if \(N \neq \Omega \setminus \textup{Max}(u)\), then 
\[
\limsup_{t \to u_{\textup{max}}^{-}} U(t) = +\infty
\]
holds, contradicting Proposition \ref{propEnergy}. We will follow \cite[Proposition 5.4]{BM}.

First, observe that since it must be \(\overline N \neq \overline \Omega\), the Lojasiewicz Structure Theorem (see \cite[Theorem 6.3.3]{Kr}) implies that \(\textup{Max}(u) \cap \partial N\) contains at least one analytic hypersurface, possibly with boundary. In fact, since 
\[
\Delta u(p) = -(1+nku_{\textup{max}}) < 0, \quad \forall p \in \textup{Max}(u),
\]
it follows from \cite[Corollary 3.4]{Chr} that \(\textup{Max}(u) \cap \partial N\) contains a complete analytic hypersurface without boundary. Therefore, if \(L^{n-1}\) denotes the \((n-1)\)-dimensional Lebesgue measure associated with the metric \(g\), then 
\[
L^{n-1}(\textup{Max}(u) \cap \partial N) > 0.
\]
Next, since \(\textup{Max}(u)\) is compact, we can apply the Lojasiewicz inequality (see, for example, \cite[Theorem 2.1]{Chr}) to find a tubular neighborhood \(V\) of \(\textup{Max} \cap \partial N\) without critical points of \(u\) and two constants \(s > 0\), \(\frac{1}{2} \leq \theta < 1\) such that 
\[
|\nabla u|(x) \geq s (u_{\textup{max}} - u)^{\theta}, \quad \forall x \in V.
\]
In particular, since \(\{u=t\} \cap \overline{N} \subset V\), if \(t \in \textup{Reg}^N(u)\) is sufficiently close to \(u_{\textup{max}}\), we have that
\begin{eqnarray*}
\frac{1}{\left( u_{\text{max}} + \frac{n k}{2} u_{\text{max}}^2 \right)^{\frac{n}{2}}} \int_{\partial \Omega \cap \overline{N}} |\nabla u| &=& U(0) \geq U(t) \\
&\geq& \frac{s}{(u_{\text{max}} - t)^{\frac{n}{2} - \theta} \left( 1 + \frac{nk}{2} (u_{\text{max}} + t) \right)^{\frac{n}{2}}} L^{n-1}(\{u=t\} \cap \overline{N}),
\end{eqnarray*}
where in the first inequality we have used Proposition \ref{propEnergy} and the second follows from the Lojasiewicz inequality. Now, observe that since \(u\) is real analytic, we can reason as in the proof of \cite[Proposition 5.4]{BM} to conclude that 
\[
\limsup_{t \to u_{\text{max}}} L^{n-1}(\{u=t\} \cap \overline{N}) > 0.
\]
Thus, since \(\theta < 1\), we arrive at a contradiction with Proposition \ref{propEnergy}, because the right-hand side of the previous chain of inequalities is unbounded. Therefore, the result follows.
\end{proof}
Now Theorem \ref{theorem-B} applies directly to the case of Einstein manifolds.
\begin{proof}[Proof of Corollary \ref{corollary-2}]
    If $(M,g)$ is Einstein, then it must be a real analytic manifold because of \cite[Theorem 5.26]{besse}. Thus, the result follows directly from Theorem \ref{theorem-B}. 
\end{proof}








\appendix
\section{An alternative approach}\label{alternative}
In this section, following a similar approach to \cite{freitas} (see also \cite{farina-roncoroni}), we alternatively obtain the Pohozaev-type formula of Proposition \ref{poho_type}. We follow the notation of earlier sections, where \(X\) is a conformal field with conformal factor \(\psi\) in a Riemannian manifold \((M^n, g)\) and \((\Omega, u)\) is a solution for \eqref{serrinproblem}. Throughout this section, we may assume \(n \geq 2\).
First of all, we calculate the expression \(\Delta \langle X, \nabla u \rangle\).
\begin{lemma}\label{eq004}Let $u$ be a solution of \eqref{serrinproblem} in a Riemannian manifold $(M^n,g)$ endowed with a conformal vector field $X$. Then
$$\Delta\langle X, \nabla u\rangle  = (3-n)\langle \nabla \psi,\nabla u\rangle  -3\psi(1+nku)-{\text{div}}(\nabla_{\nabla u}X)-\displaystyle\sum\langle\nabla \varphi(E_i,E_i),\nabla u\rangle -nk\langle \nabla u,X\rangle,$$
where $\left\{E_1, \dots, E_n\right\}$ is an orthonormal frame in a neighbourhood of a point, $\psi$ is the conformal factor of $X$ and $\varphi$ is the skew-symmetric tensor defined in \eqref{eq002}.
\end{lemma}

\begin{proof}
First, we calculate the gradient of $\langle X, \nabla u\rangle.$ In fact, by considering $Y\in \mathfrak{X}(M),$ we obtain
\begin{eqnarray*}
\langle \nabla \langle X,\nabla u\rangle,Y\rangle &=& Y\langle X,\nabla u\rangle\\
&=&\langle \nabla_Y X,\nabla u\rangle+\langle X,\nabla_Y\nabla u\rangle\\
&=&2\psi g(Y,\nabla u)-\langle Y,\nabla_{\nabla u} X\rangle+\nabla^2u(X,Y),
\end{eqnarray*}
where we have used that $\mathcal{L}_Xg(Y,\nabla u)=2\psi g(Y,\nabla u)=\langle\nabla_YX,\nabla u\rangle+\langle Y,\nabla_{\nabla u}X\rangle.$
Therefore,
$$\nabla \langle X, \nabla u\rangle =2\psi\nabla u -\nabla_{\nabla u}X+\nabla^2u(x,\cdot).$$
This implies that,
\begin{eqnarray*}
\Delta \langle X, \nabla u\rangle &=& \text{div}(\nabla\langle X,\nabla u\rangle)\\
&=&2\psi\Delta u+2\langle \nabla \psi,\nabla u\rangle-\text{div}(\nabla_{\nabla u}X)+\dfrac{1}{2}\langle \nabla^2u,\mathcal{L}_Xg \rangle+\text{div}(\nabla^2u(X,\cdot)).
\end{eqnarray*}
Given $p \in M$, take an orthonormal frame $\left\{E_1, \dots, E_n\right\}$ in a neighbourhood of $p$. Then, using that $\Delta u = -nku -1,$ $\mathcal{L}_Xg=2\psi g,$  $\text{div}(\nabla^2u(X,\cdot))=\langle \nabla \Delta u, X\rangle +Ric(X,\nabla u)$ and \eqref{eq003}, we infer
\begin{eqnarray*}
\Delta \langle X, \nabla u\rangle 
&=&3\psi(-1-nku)+2\langle \nabla \psi,\nabla u\rangle -\text{div}(\nabla_{\nabla u}X)+\langle \nabla \Delta u, X\rangle +Ric(X,\nabla u)\\
&=&3\psi(-1-nku)+2\langle \nabla \psi,\nabla u\rangle -\text{div}(\nabla_{\nabla u}X)-nk\langle \nabla u, X\rangle -(n-1)\langle \nabla \psi, \nabla u \rangle \\
&-&\langle \displaystyle \sum_{i=1}^n\nabla\varphi(E_i,E_i),\nabla u\rangle\\
&=&(3-n)\langle \nabla \psi,\nabla u\rangle -3\psi(1+nku)-\text{div}(\nabla_{\nabla u}X)-\displaystyle\displaystyle \sum_{i=1}^n\langle\nabla \varphi(E_i,E_i),\nabla u\rangle-nk\langle \nabla u,X\rangle.
\end{eqnarray*}
\end{proof}

\begin{remark}
When \(X\) is a closed conformal field, \(\varphi \equiv 0\) and 
\(\operatorname{div}(\nabla_{\nabla u} X) = \langle \nabla \psi, \nabla u \rangle + \psi \Delta u\), so that we obtain again \cite[Lemma 2]{freitas}.
\end{remark}

The next result provides a Pohozaev-type identity as in Proposition \ref{poho_type}.
\begin{lemma}[Pohozaev-type identity] \label{lemmapoho}Under hypothesis of Lemma \ref{eq004}, we have
 \begin{eqnarray*}
\frac{n-2}{2n(n-1)}\displaystyle\int_{\Omega}u^{2}\left(\frac{X(R_{g})}{2}+\psi R_g\right)dV_{g}+\frac{n+2}{n}\displaystyle\int_{\Omega}u\psi  dV_{g}+2k\displaystyle\int_{\Omega}u^{2}\psi dV_{g}=c^{2}\int_{\Omega}\psi dV_{g}. \end{eqnarray*}   
\end{lemma}
\begin{proof}
In the proof of this lemma, we will work with a fixed family of vector fields over \( \Omega \) to carry out some local computations. Since \( \Omega \) is compact, there exists a collection of bounded domains \( U_1, \dots, U_m \subset M \) such that \( U_{[i]} \cap U_{[i+1]} \neq \emptyset \) for each \( [i] = i \, (\textup{mod } m) \), with \( \Omega \subset \bigcup_{i=1}^m U_i \) and an orthonormal frame $\left\{E^{i}_1, \dots, E^{i}_n\right\}$ defined on each \( U_i \) for \( i \in \{1, \dots, m\} \). Then, for each \( j \in \{1, \dots, n\} \), we define the vector field
\begin{equation}\label{orthonormalSystem}
    E_j (p) =	\left\{	\begin{matrix}
		 E^1_j (p) & \text{ if } & p \in U_1 ,\\[2mm]
			E^{i}_j (p) & \text{ if } & p \in U_i \setminus \left( \bigcup_{l=1}^{i-1} U_l \right), \textup{ for } i =2, \dots, m.
		\end{matrix}\right.
\end{equation}
Although the fields \( E_i \) may not be continuous in \( \Omega \), for any point \( p \in M \), there exists a neighborhood of \( p \) in which the family $\left\{E_1, \dots, E_n\right\}$ defines an orthonormal frame. Thus, by Lemma \ref{eq004}, we obtain that
\begin{eqnarray}\label{eq005}
u\Delta\langle X, \nabla u\rangle-\langle X,\nabla u\rangle \Delta u&=&3u\psi(-1-nku)-(n-3)u\langle \nabla \psi,\nabla u\rangle\nonumber \\
&-&u\text{div}(\nabla_{\nabla u}X)-u\langle \displaystyle \sum_{i=1}^n\nabla\varphi(E_i,E_i),\nabla u\rangle+\langle X,\nabla u\rangle.
\end{eqnarray}
Now, observe that the quantity \( \left\langle \sum_{i=1}^n \nabla \varphi(E_i, E_i), \nabla u \right\rangle \) can be interpreted as the trace of a \((2,0)\)-tensor over \( \Omega \), making it a smooth function on \( \Omega \), despite the vector fields \( E_i \) not being continuous throughout the whole of \( \Omega \). Therefore, we can integrate the equation \eqref{eq005} over \( \Omega \) using the classical Divergence Theorem. First, we analyze the left-hand side of this identity. Observe that
$$u\Delta\langle X, \nabla u\rangle-\langle X,\nabla u\rangle \Delta u = \text{div}(u\nabla\langle X,\nabla u\rangle - \langle X,\nabla u\rangle \nabla u).$$
Since $u=0,$ $\nu=-\dfrac{\nabla u}{|\nabla u|}$ and $|\nabla u|=c$ on $\partial \Omega,$ we get 
\begin{eqnarray}\label{eq006}
\displaystyle\int_{\Omega}(u\Delta\langle X, \nabla u\rangle-\langle X,\nabla u\rangle \Delta u)dV_g = c \displaystyle\int_{\partial \Omega}\langle X,\nabla u\rangle d\sigma_g.
\end{eqnarray}
Furthermore, since \( \text{div} X = n \psi \), because \( X \) is a conformal field with conformal factor \( \psi \), we have
\begin{eqnarray}\label{eq007}
\displaystyle\int_{\partial \Omega}\langle X,\nabla u\rangle d\sigma_g&=& -\displaystyle\int_{\partial \Omega} \langle X, \nu\rangle|\nabla u|d\sigma_g\nonumber\\
 &=&-c\displaystyle\int_{\partial \Omega} \langle X, \nu\rangle d\sigma_g=-c\displaystyle\int_{ \Omega}\text{div} X dV_g=-cn\displaystyle\int_{ \Omega} \psi dV_g.
\end{eqnarray}
Combining \eqref{eq006} and \eqref{eq007}, we infer that
\begin{eqnarray}\label{eq008}
\displaystyle\int_{\Omega}(u\Delta\langle X, \nabla u\rangle-\langle X,\nabla u\rangle \Delta u)dV_g= -c^2n\displaystyle\int_{ \Omega} \psi dV_g.
\end{eqnarray}
Second, after integrating over \(\Omega\) in \eqref{eq005}, we intend to calculate all terms on the right side of this expression in the same way as \cite{freitas}. Note that,
\begin{eqnarray}\label{eq009}
\displaystyle\int_{\Omega}\langle X, \nabla u\rangle dV_g = - \displaystyle\int_{\Omega} u \text{div}(X) dV_g = -n\displaystyle\int_{\Omega} u\psi dV_g,
\end{eqnarray}
where we have used that $\text{div}(uX)=u\text{div}(X)+\langle \nabla u,X\rangle$ and $u=0$ on $\partial \Omega.$
Moreover, 
\begin{eqnarray}\label{eq010}
\displaystyle\int_{\Omega}u\langle \nabla \psi, \nabla u\rangle dV_g= -\dfrac{1}{2}\displaystyle\int_{\Omega}u^2\Delta \psi dV_g,
\end{eqnarray}
because $\text{div}(u^2\nabla \psi)=u^2\Delta\psi+2u\langle \nabla u,\nabla \psi\rangle$  and $u=0$ on $\partial \Omega.$
We recall that from Lemma \ref{lemma 2.1} we have 
\begin{eqnarray}\label{eq011}
-(n-1)\Delta \psi = \dfrac{1}{2}X(R_g)+\psi R_g.
\end{eqnarray}
By equations \eqref{eq010} and \eqref{eq011}, we infer 
\begin{eqnarray}\label{eq012}
\displaystyle\int_{\Omega}u\langle\nabla u, \nabla \psi\rangle dV_g= \dfrac{1}{2(n-1)}\displaystyle\int_{\Omega}u^2\left(\dfrac{1}{2}X(R_g)+\psi R_g\right)dV_g.
\end{eqnarray}
Moreover, $\text{div}(u\nabla_{\nabla u}X)=u \text{div}(\nabla_{\nabla u}X)+\langle \nabla u, \nabla_{\nabla u}X\rangle=u \text{div}(\nabla_{\nabla u}X)+\psi|\nabla u|^2.$
In last equality, we use that $\mathcal{L}_Xg(\nabla u,\nabla u)=2\langle \nabla_{\nabla u}X,\nabla u\rangle=2\psi|\nabla u|^2. $ This implies that,
\begin{eqnarray}\label{eq013}
\displaystyle\int_{\Omega}u\text{div}(\nabla_{\nabla u}X)dV_g=-\displaystyle\int_{\Omega}\psi|\nabla u|^2 dV_g.
\end{eqnarray}
We note that
\begin{eqnarray}\label{eq014}
\displaystyle\int_{\Omega}u\langle\displaystyle\sum \nabla \varphi(e_i,e_i),\nabla u\rangle dV_g=-\dfrac{1}2\displaystyle\int_{\Omega}u^2\text{div}\left(\displaystyle \sum_{i=1}^n\nabla\varphi(e_i,e_i)\right)dV_g=0,
\end{eqnarray}
where we have used that $\text{div}\left(u^2\sum(\nabla\varphi(e_i,e_i)\right)=u^2\text{div}(\sum\nabla\varphi(e_i,e_i))+2u\langle\nabla u, \sum\nabla \varphi(e_i,e_i) \rangle$ and that $\text{div}\left(\displaystyle \sum_{i=1}^n\nabla\varphi(e_i,e_i)\right)=0$ because of \eqref{div_0}. Finally, joining the terms, we conclude that
\begin{eqnarray}\label{prev_poho}
\dfrac{(n+3)}{n}\int_{\Omega}u\psi   dV_g&=&c^2\displaystyle\int_{\Omega}\psi dV_g+\dfrac{(3-n)}{2n(n-1)}\displaystyle\int_{\Omega}u^2\left(\dfrac{X(R_g)}{2}+\psi R_g\right)dV_g\nonumber\\
& &-3k\displaystyle\int_{\Omega}\psi u^2 dV_g+\dfrac{1}{n}\displaystyle\int_{\Omega}\psi|\nabla u|^2 dV_g.
\end{eqnarray}
To obtain the expression as stated in the lemma, it is sufficient to focus on the last integral term. Indeed, by applying the Divergence Theorem, we have
\begin{eqnarray*}
\frac{1}{n}\int_{\Omega}\psi |\nabla u|^{2}dV_g &=& \frac{1}{n}\left[-\int_{\Omega} \psi u \Delta u + \frac{1}{2}\int_{\Omega} u^{2} \Delta \psi \right]dV_g \\
&=& \frac{1}{n} \int_{\Omega} \psi u dV_g + k \int_{\Omega} \psi u^{2} dV_g - \frac{1}{2n} \int_{\Omega} u^{2} \Delta \psi dV_g \\
&=& \frac{1}{n} \int_{\Omega} \psi u dV_g+ k \int_{\Omega} \psi u^{2} dV_g- \frac{1}{2n(n-1)} \int_{\Omega} u^2 \left( \frac{X(R_g)}{2} + \psi R_g \right)dV_g,
\end{eqnarray*}
where we used \eqref{serrinproblem} and Lemma \ref{lemma1}. Substituting this into \eqref{prev_poho}, we obtain the desired expression.
\end{proof}

\section*{Data availability statement}
This manuscript has no associated data.

\section*{Conflict of interest statement}
On behalf of all authors, the corresponding author states that there is no conflict of interest.

\section*{Acknowledgments}
The authors would like to thank Luciano Mari for him discussion about the object of this paper and several valuable suggestions.

The first and second authors would like to thank the hospitality of the Mathematics Department of  Università degli Studi di Torino, where part of this work was carried out (there, Andrade was suported by CIMPA/ICTP Research in Pairs Program 2023 and Freitas was supported by CNPq/Brazil Grant 200261/2022-3). 
The first author also has been partially supported by Brazilian National Council for Scientific and Technological Development (CNPq Grant 403349/2021-4 and CNPq Grant 408834/2023-4) and FAPITEC/SE/Brazil (Grant 019203.01303/2021-1). 

The second author has also been partially supported by CNPq/Brazil Grant 316080/2021-7, by the public call n.~03 Produtividade em Pesquisa proposal code PIA13495-2020 and Programa Primeiros Projetos, Grant 2021/3175-FAPESQ/PB and MCTIC/CNPq.

The third author is partially supported by the {\it Maria de Maeztu} Excellence Unit IMAG, reference CEX2020-001105-M, funded by MCINN/AEI/10.13039/ 501100011033/CEX2020-001105-M and Spanish MIC Grant PID2020-117868GB-I00.

\end{document}